\documentclass[a4paper,11pt,reqno]{amsart}
\usepackage{graphicx}
\usepackage[english]{babel}

\usepackage{amssymb,enumerate,amsmath}
\usepackage[bookmarks,colorlinks, linkcolor=green]{hyperref}
\usepackage{tikz}
\usepackage{array}
\usepackage{mathrsfs}
\usepackage[latin1]{inputenc}
\usepackage[T1]{fontenc}
\usepackage{ae,aecompl}

\usepackage{ulem}

\usepackage{mathdots}

\usepackage[margin=2.5cm]{geometry}

\DeclareMathOperator{\Spec}{Spec}

\newcommand{\ZZ}{\mathbb{Z}}

\newcommand{\PP}{\mathbb{P}}
\renewcommand{\AA}{\mathbb{A}}

\DeclareMathOperator{\Proj}{Proj}
\DeclareMathOperator{\tr}{tr}

\renewcommand{\geq}{\geqslant}
\renewcommand{\leq}{\leqslant}

\newcommand{\Pl}{Pl\"ucker}

\newcommand{\QUOT}{\mathrm{Quot}^d_{\mathcal  F^r/  \mathbb P^1/  k}}

\newcommand{\QUOTf}{{Quot}^d_{\mathcal  F^r/  \mathbb P^1/   k}}

\newcommand{\QUOTZf}{{Quot}^t_{\mathcal  F_Y/Y  /   k}}
\newcommand{\QUOTZ}{\mathrm{Quot}^t_{\mathcal  F_Y/ Y /   k}}

\newcommand{\HILB}{\mathrm{Hilb}^d_{\mathcal O_{\PP^1}}}
\newcommand{\HILBf}{Hilb^{d}_{\mathcal O_{\PP^1}}}

\newcommand{\matrixPM}{P(M)}
\newcommand{\matrixPP}{\mathbf P}
\newcommand{\Ann}{\mathrm{Ann}}

\numberwithin{equation}{section}
\newtheorem{theorem}{Theorem}[section]
\newtheorem{corollary}[theorem]{Corollary}
\newtheorem{proposition}[theorem]{Proposition}
\newtheorem{lemma}[theorem]{Lemma}
\newtheorem{remark}[theorem]{Remark}

\theoremstyle{definition}
\newtheorem{definition}[theorem]{Definition}
\newtheorem{example}[theorem]{Example}
\newtheorem{notation}[theorem]{Notation}

\begin{document}

\title{On the Quot scheme $\mathrm{Quot}_{\mathcal O_{\mathbb P^1}^r/\mathbb P^1/k}^d$}
\author[C. Bertone]{Cristina Bertone}
\author[S. L. Kleiman]{Steven L. Kleiman}
\author[M. Roggero]{Margherita Roggero}

\address{Cristina Bertone \and Margherita Roggero\\ Dipartimento di Matematica, Universit\`a degli Studi di Torino \\ Via Carlo Alberto 10 \\ 10123 Torino \\ Italy}
\email{\href{mailto:cristina.bertone@unito.it}{cristina.bertone@unito.it}, \href{mailto:margherita.roggero@unito.it}{margherita.roggero@unito.it}}

\address{Steven L. Kleiman\\ Dept. of Math., 2-172 MIT, 77 Mass. Ave., Cambridge, MA 02139, USA }
\email{\href{mailto:kleiman@math.mit.edu}{Kleiman@math.MIT.edu }}

\begin{abstract}
 We consider the quot scheme $\QUOT$  of locally free    quotients  of   $\mathcal F^r:= \bigoplus ^{ r} \mathcal O_{\mathbb P^1 }$ with  Hilbert polynomial $p(t)=d$.   We prove that it is a smooth variety of dimension $dr$, locally isomorphic to $\AA^{dr}$.
We   introduce a new notion of support for modules in $\QUOT$, called  {\it Hilb-support}   that allows us  to    define a natural surjective morphism of schemes $\xi  :\QUOT \to \HILB $   associating  to each module  its Hilb-support and study the  fibres  of $\xi$ over each $k$-point $Z$ of $\HILB$. If $Z=Y_1+\dots+Y_n$, with $Y_j=t_jR_j$, where $R_1, \dots, R_n$ are distinct points, the fibre of $\xi$ over $Z$ is isomorphic to   $\mathrm{Quot}^{t_1}_{\mathcal  F\otimes \mathcal O_{Y_1}/ Y_1/   k}\times\dots  \times \mathrm{Quot}^{t_n}_{\mathcal  F\otimes \mathcal O_{Y_n}/ Y_n/   k}$. 
  We then study the  Quot scheme $\mathrm{Quot}^{t}_{\mathcal  F^r\otimes \mathcal O_{Y}/ Y/   k}$  with  $Y=tR$. For $t=1$, $\QUOTZ$ is isomorphic to $\PP^{r-1}$, while  for $t\geq 2$ we prove that    it is formed by  a main irreducible, reduced and singular component of dimension $t(r-1)$ and by some embedded component of lower dimension.
\end{abstract}


\maketitle
\section{Introduction}

%

Let $k$ be an  algebraically closed field of any characteristic. 
We fix the integers $r\geq 2$ and $d\geq 1$ and  set  $S:=k[x,y]$,  $\PP^1:=\Proj(S)$, 
$F^r := \bigoplus^r S$ and   $\mathcal F^r := \bigoplus^r \mathcal O_{\PP^1}$.

In the present paper we investigate the Quot scheme  parameterizing locally free quotient modules of $F^r$ having Hilbert polynomial $d$.
We refer to Grothendieck \cite{Gr1961} for the definition and construction  of the Quot scheme $\QUOT$. 
 Grothendieck's  construction of the Quot scheme 
yields $\mathrm{Quot}^d_{\mathcal F^r/\mathbb P^1/k}$ as a closed subscheme of the Grassmannian $\mathbb G_k(d, (d+1)r)$. For a more detailed treatment, see  \cite{nitsure}.

In the following, for sake of simplicity, we will identify points of $\QUOT$    with the corresponding quotient modules and points of $\HILB$ with the corresponding subschemes of $\PP^1$.

The only completely already known   case is the one  with $d=1$:  in \cite[Proposition 2.2]{kle} S. L. Kleiman proved that $\mathrm{Quot}^1_{\mathcal  F^r/  \mathbb P^1/   k}$ is isomorphic to $\PP(\mathcal F^{r})$ .

We prove that $\QUOT$ is smooth, rational of dimension $dr$ (Corollary \ref{cor:smooth}) and it is covered by Zariski open affine  subsets (Proposition \ref{star}). For each of them, we describe an explicit isomorphism to $\mathbb A^{dr}$  (Corollary \ref{cor:isom}).

Although  the above quoted results   describe   $\QUOT$ as a uniform object,    the geometry of $\QUOT$ becomes   more entangled as soon as one takes into account the algebraic structure of the modules corresponding to its points.

A first aspect that makes  some 
points of $\QUOT$  substantially different from   others  is a significant diversity between their schematic supports:    for instance,  a module 
whose schematic support is made  of simple points is certainly not isomorphic to one whose schematic support is a multiple structure over a point. In order to deepen  this observation and exploit it to study $\QUOT$ , we   define  a natural surjective map $\xi  :\QUOT \to \HILB $. The starting idea was indeed to associate to a module its schematic support.  
 However, it is easy to find modules in $\QUOT$ whose schematic support has length $<d$. This is why we introduce the notion of {\it Hilb-support} for the points of $\QUOT$ (Definition \ref{def:hilbsupp}), and use this notion to define the map $\xi$ (Definition \ref{def:xi}).

We then investigate the fibers of $\xi$, by studying Quot schemes $\QUOTZ:=\mathrm {Quot}^t_{\mathcal  F\otimes \mathcal O_Y/ Y /   k}$ over a $0$-dimensional scheme  $Y$  in $\PP^1$,  with a special attention to the case $Y=tR$, $R\in \mathbb P^1$, the one with the main  interest for the description of  the fibers of $\xi$.   For $t=1$ the fibre is isomorphic to $\PP^{r-1}$, while for
 $t\geq 2$, we prove that 
 $\QUOTZ$ is  a scheme formed    by a main component  of dimension $t(r-1)$, which is irreducible and reduced but singular, and  by embedded components of lower dimension   (Theorem \ref{thm:sing}).

 Exploiting the above result we obtain 
the main  result of the paper:
\begin{theorem}  \label{thm:main}
Let $Y=Y_1+ \dots +Y_n\in \HILB$ with $Y_j=t_jR_j$ and      $R_1, \dots, R_n$ distinct points.  The fiber of $\xi$ over $Y$  is  isomorphic to the product  ${Quot}^{t_1}_{\mathcal  F_{Y_1}/ Y_1/   k}\times \dots \times  {Quot}^{t_n}_{\mathcal  F_{Y_n}/ Y_n/   k} $. \\ 
If $n=d$, then   the fiber of $\xi$ over $Y$ is isomorphic to the product of $d$ copies of  $\mathbb P^{r-1}$. If $n\neq d$, then  the fiber of $\xi$ over $Y$ is irreducible and generically smooth, but has embedded components and its support  has   singular points.
\end{theorem}

Several statements of the present paper  have a local nature, hence in the proofs we can restrict ourselves to an open subset. We choose a   open cover made of  subsets of $\QUOT$ given by the non-vanishing of suitable \Pl\ coordinates.  On each of these open subsets,   a  module $M$ of $\QUOT$ is characterized by a $d\times d$ matrix $\matrixPM$  that encodes its structure of $k[x,y]$-module (Proposition \ref{basisx}). This matrix $\matrixPM$ is central in all the present paper.   For instance,  the definition of Hilb-support of a module in $\QUOT$, and hence the definition of $\xi$, rely on the characteristic polynomial of this matrix $\matrixPM$.   Furthermore,  up to a linear change of coordinates in $\PP^1$,  $\QUOTZ$  with $Y=tR$ can be locally described by the condition $\matrixPM^t=0$;  we prove Theorem \ref{thm:sing}   by the consequent  local description of  $\QUOTZ$    as a scheme of nilpotent matrices, hence a subscheme of the so-called {\it null-cone} (see for some details  Lemma \ref{varw} and Example \ref{ex:[1,..,1]}).


The arguments we just outlined are developed below not only for
$k$-points, but for families, as is required for working with the
functor of points of the Quot scheme.

The starting point for the results of the present paper are the computational results exposed in \cite{ABRS}: although here we do not directly use those methods, the choice of a  suitable  open cover and the explicit computation of  local and global equations for   the Quot scheme with   $t=d=2$ gave us inspiration and conjectures for the more general case that we  prove in the present paper.

\section{Notation and general setting}\label{ann}

Let $k$ be an  algebraically closed field of any characteristic.   In the following $A$ will always denote a  $k$-algebra.

We fix the integers $r\geq 2$ and $d>0$ and  set  $S_A:=A[x,y]$,  $\PP^1_A:=\Proj(S_A)$, 
$F^r_A := \bigoplus^r S_A$ and   $\mathcal F^r_{A} := \bigoplus^r \mathcal O_{\PP^1_{ A }}$.
If $A=k $, we will simply write $S$, $\PP^1$,      $F^r$, $\mathcal F^r$. 

We will denote by $\mathbb G$  the Grassmannian  $\mathbb G_k (d,(d+1)r)$ and by $G$ its functor of points.

We refer to Grothendieck \cite{Gr1961} for the definition of the Quot scheme $\QUOT$ and the Quot functor $\QUOTf$. 
The Gotzmann number of 
the Hilbert polynomial $p(t)=d$ is $d$ \cite{Dellaca}, hence every  saturated quotient module $F^r_A/V\in \QUOTf{(A)}$ is uniquely determined by its  homogeneous component of degree
 $d$, namely  by $(F_A^r/V)_d=(\bigoplus^r S_{A,d})/V_{d}$. 
 
In this way we obtain  the standard embedding  of $\QUOT$ in  $\mathbb G$ and then in $\PP^{N\!(r)}$  where $N\!(r)={(d+1)r\choose d}-1$  \cite[Section 7]{ABRS}.     Taking in consideration this embedding, in the following we often identify every quotient module $F_A^r/V$ in   $  \QUOTf{(A)}$   with its degree $d$ homogeneous component, namely by the  rank $d$ locally free $A$-module   $M:=F_{A,d}/V_d$.

 We recall that the support of a $B$-module $D$ is the subscheme of $\Spec (B)$ defined by the ideal $\Ann(D):=\{b \in B $ such that $bD=0\}$. If $A$ is a $k$-algebra and $F_A^r/V\in 
\QUOTf(A)$, then  
   $\Ann_{S_A} (F_A^r/V) \cap A=(0)$ and $\Ann_{S_A} (F_A^r/V) \not \supset (x,y)$, because $(F_A^r/V)_m$ is a locally free module of rank $d$ for every $m\geq d$. Therefore, the  support of 
$M=(F_A^r/V)_d$  is a closed subscheme of $\PP^1_A$ of length at most $d$.

Let $e_1, \dots ,  e_r$ be the standard basis of $F_A^r$ as an $S_A$-module.
We choose the following ordered list 
$$\mathcal B:=[x^d e_1,x^{d-1} y e_1,\dots, y^d e_1, \dots ,x^d e_r,x^{d-1} y e_r,\dots, y^d  e_r]$$ as a basis for 
 $\bigoplus^r S_{A,d}$ as an $A$-module; $\mathcal B_j$ will denote the $j$-th element of this list.

Every free  (not simply locally free)  quotient  module  $M\in  G(A)$   can be represented by a $d\times (d+1)r$ matrix $\mathbf C(M)$, where the rows correspond to the choice of a $A$-basis $u_1,\dots, u_d$ for 
 $M$. 
 Let us denote by  $C_i=(c_{1,i}\dots c_{d,i})^T, i=1,\dots, (d+1)r$, the columns of this matrix; then 
 $$\mathbf C(M)=[C_1 \ C_2 \dots \ C_{d+1}\ \vert\  \dots  \  \vert \ C_{(d+1)r-d}\ C_{(d+1)r-d+1}\ \dots \ C_{(d+1)r}].$$

Considering  the  \Pl\ embedding  of $\mathbb G$    in $\PP^{N\!(r)}=\Proj(\ZZ[\Delta_{j_1,\dots,j_d}]_{1\leq j_1<\cdots <j_d\leq (d+1)r})$, 
 the minors of  $\mathbf C(M)$  are the \Pl\ coordinates of $M$:  for every $1\leq j_1<\cdots <j_d\leq (d+1)r$,  we set  
$\Delta_{j_1,\dots,j_d}(M)=\det(C_{j_1}\dots \ C_{j_d})$.

We will call the $v$-th block of $\mathcal B$ (resp.  of $\mathbf C(M)$) the subset (resp. submatrix)  with indices  from  $(d+1)v-d$ to $ (d+1)v$, namely those corresponding to $e_{v}$.
When we want  to emphasize the  division in blocks of the matrix $\mathbf C(M)$  we will write it as 
   $$\mathbf C(M)=[B_{1,1} \ B_{1,2} \dots \ B_{1,d+1}\ \vert\  \dots  \  \vert \ B_{r,1} \ B_{r,2} \dots \ B_{r,d+1}],$$
so that $B_{\ell ,j}$is the $j$-th column of the $\ell $-th block.

If $A$ is not a local ring, then not every module $M \in G(A)$ can
be represented by such a matrix; indeed, $M$ must be free (in general,
$M$ is only locally free). For this reason, it is useful to work on open
sets that cover $\mathbb G$ and on which every $M$ is free.

\section{A suitable open cover of $\mathbb G$ and of $\QUOT$}

 For every  $i,j$ with $1\leq j_1<\cdots <j_d\leq (d+1)r$,  we will denote by $ \mathbb G_{j_1,\dots,j_d}$ the open subscheme of $\mathbb G$ where the \Pl\ coordinates  $\Delta_{j_1,\dots,j_d}$ is invertible and by $G_{j_1,\dots,j_d}$ the corresponding open subfunctor of $G$.

 Similarly,  we will denote 
 by $\mathbb Q_{j_1,\dots,j_d}$ the open subscheme $\mathbb G_{j_1,\dots,j_d}\cap \QUOT$ of $\QUOT$ and by 
$Q_{j_1,\dots,j_d}$  the corresponding open subfunctor of $\QUOTf$.

 Every   $A$-module $M\in G_{j_1,\dots,j_d}(A)$  is free with basis   $\mathcal B_{j_1}, \dots, \mathcal B_{j_d}$.  Choosing this basis as $u_1, \dots, u_d$,  the submatrix of $\mathbf C(M)$  corresponding to columns ${j_1,\dots,j_d}$ is the identity matrix. 
 
  In particular, we consider a list of positive integers $[i_1,\dots i_s]$ with $i_1\geq \cdots\geq i_s>0$ and $\sum i_m=d$ and  denote by $\mathbb G_{[i_1,\dots,i_s]}$ and $\mathbb Q_{[i_1,\dots,i_s]}$  the open subsets  obtained   selecting as  ${j_1,\dots,j_d}$ the first $i_m$ indices  of  the $m$-th block for every  $m=1, \dots, s$. 
 \begin{example}
  Consider $r=3$, $d=4$. If we choose $[i_1,i_2]=[2,2]$, this corresponds to $j_1=1, j_2=2, j_3=6, j_4=7$.
  \end{example}
The  $N\!(r)$ open subschemes $ \mathbb G_{j_1,\dots,j_d}$ form the standard open cover of $ \mathbb G$ and none of them is contained in the 
union of the other ones.  For a general introduction to such open covers see \cite[Section 9.8]{GD}.  
 The $ \mathbb G_{j_1,\dots,j_d}$ are all isomorphic, but
the $\mathbb Q_{[i_1,\dots,i_s]}$ are not. This is due to the fact that  the modules $M$ corresponding to the points of $\mathbb Q_{[i_1,\dots,i_s]}$ have a specific module structure related to $[i_1,\dots,i_s]$. 

So we will consider transformations by suitable changes of coordinates
in $\PP^1$ and suitable change of the basis $\mathcal B$,
transformations that preserve the scheme structure of the $\QUOT$ and
that take quotient modules representing points of $\QUOT$ to isomorphic
quotients.

 Below, the variables $x,y$ and the basis $e_1,\dots, e_r$ are changed
via matrices in $\mathrm{GL}_2(k)$ and $\mathrm{GL}_r(k)$.
 
\begin{notation}
 We will denote by $\mathbb G_{j_1,\dots,j_d} ^\ast$ and $\mathbb
G_{[i_1,\dots,i_s]}^\ast$ (resp., $\mathbb Q_{j_1,\dots,j_d}^\ast$ and
$\mathbb Q_{[i_1,\dots,i_s]}^\ast $) the union of the $\mathbb
G_{j_1,\dots,j_d} $ and $\mathbb G_{[i_1,\dots,i_s]}$ (resp., $\mathbb
Q_{j_1,\dots,j_d}$ and $\mathbb Q_{[i_1,\dots,i_s]}$), up to the linear
changes described above.

Moreover, we will denote by $G^\ast$ and $Q^\ast$ the union of the open subfunctors
 $ Q_{[i_1,\dots,i_s]}  $ and $ G_{[i_1,\dots,i_s]}$, for every  sequence $[i_1,\dots,i_s]$ and up to the above linear changes.
\end{notation}

\begin{remark} Observe that   $Q^\ast(A)=\QUOTf(A)$ for every  local $k$-algebra $A$.
\end{remark}

\begin{proposition} \label{star}
In the above setting, we have 

\[\QUOT= \bigcup_{i_1\geq \cdots\geq i_s>0, \sum i_m=d} \mathbb Q_{[i_1,\dots,i_s]} ^\ast .\]
\end{proposition}

\begin{proof} As it suffices to check the asserted equation locally, we consider a local $k$-algebra $(A, \mathfrak m, K)$.
It is sufficient to prove that for every $M\in \QUOTf(A)$, there is $[i_1,\dots,i_s]$ such that $M\in  Q_{[i_1,\dots,i_s]} ^\ast(A)$. Since we will look for invertible submatrices of $\mathbf C(M)$, we can work modulo $\mathfrak m$, namely take $A=K$ a field.

  Up to a change of coordinates in $\mathrm{GL}_2(k)$, we may assume that $M$ is supported on the open subset $\mathbb A^1_K$ of $\PP^1_K$  where $x$ does not vanish, so that $xM \hookrightarrow M$ .

 Then, after reordering $e_1,\dots,e_r$, we may  assume that  in the first block of $\mathcal B$  there is an element $x^{d-s}y^s e_1$ that does not vanish  in $M$, hence  we may also assume that $x^de_1$ does not vanish in $M$ (recall that in our setting $x$ is not a zero-divisor).  
Let $L_1$ be the $K$-vector space in $M$   generated by  $x^de_1,x^{d-1} ye _1,\dots, y^d e_1$, and let $i_1$ be its dimension;  by applying the same argument as above, we may choose    $x^de_1,x^{d-1} ye _1,\dots, x^{d-i_1+1}y^{i_1-1} e_1$ as its basis. Indeed, if $x^{d-s}y^{s} e_1$ is linearly dependent on $x^d e_1, \dots , x^{d-s}y^{s} e_1$, then the same holds for $x^{d-t}y^{t} e_1$ for every $t\geq s$.

 If $i_1=d$, then $M\in Q_{[d]}(K)$. Otherwise, after reordering $e_2,\dots,e_r$ we may  assume that   the $K$-vector  space  $L_2\subset M$    generated by   the elements of the first and second block of $\mathcal B$  has dimension $i_1+i_2$ larger than $i_1$ and see that $x^de_1,\dots, x^{d-i_1+1}y^{i_1-1} e_1, x^de_2, \dots, x^{d-i_2+1}y^{i_2-1} e_2$ is its basis. Up to reordering $e_1,\dots,e_s$, we obtain $i_1\geq \cdots \geq i_s>0$. Finally, we take as  basis $u_1,\dots, u_d$ of $M$ exactly the elements $x^{d}e_1, x^{d-1}y e_1,\dots, x^{d-i_1+1}y^{i_1-1} e_1, \dots, x^{d}e_s, \dots,  x^{d-i_s+1}y^{i_s-1} e_s$. Thanks to these choices, $\mathbf C(M)$ has an identity submatrix, that we will denote by  $\mathrm{Id}_{[i_1,\dots,i_s]} $,  obtained by taking the first $i_m$ columns from the $m$-th block for $m=1, \dots, s$.
\end{proof}

\begin{corollary}
A local property holds everywhere on 
 $\QUOT$ if, and only if,  it holds on the open subsets  $  \mathbb Q_{[i_1,\dots,i_s]}  $, for every $[i_1,\dots,i_s]$ such that $i_1\geq \cdots\geq i_s>0$, $\sum i_m=d$.
\end{corollary}

\section{An explicit description of   $\mathbb Q_{[i_1,\dots,i_s]}$}

Now we give an explicit characterization of the the elements of  $  Q_{[i_1,\dots,i_s]}(A)$.

\begin{proposition}\label{basisx} Let $A$ be a $k$-algebra. 
 Let $M$  be a free   $A$-module   of  $G(A)$ and  $F_A^r/V$ the  $S_A$-module generated by $M$.
 Then 
$M\in \QUOTf(A) $ if,  and only if, 
 $\exists $   $u_1, \dots, u_d\in M$,    $\ell,\ell' \in    k[x,y] _1$, and a  $d\times d$ matrix  $\matrixPM$  with entries in $A$ such that   
\begin{enumerate}[\quad \rm 1)]
\item\label{basis1}  $ u_1,\dots, u_d$ are  $A$-basis for   $(F_A^r/V)_d$;
 \item\label{basis2}   $\ell, \ell' $ are a basis for the linear forms in $  k[x,y] $, such  that $\ell u_1,\dots, \ell u_d$ are  $A$-basis  of    $(F_A^r/V)_{d+1}$;
\item  \label{basis3} $ \left[\  \ell' u_1  \cdots   \ell'u_d\  \right] =\left[\  \ell u_1  \cdots \ell u_d\  \right] \matrixPM$.
\end{enumerate}

If the above conditions are fulfilled, then  multiplication by $\ell\colon (F_A^r/V)_{\geq d}\rightarrow (F_A^r/V)_{\geq d+1} $ is injective and $\ell^{m} u_1, \dots, \ell^{m} u_d$  is  $A$-basis of $(F_{A}^r/V)_{m+d}$   for  every $m\geq 1$ .
\end{proposition}

\begin{proof}
It is obvious that the  conditions \ref{basis1}), \ref{basis2}), \ref{basis3}) are sufficient to ensure that   $M$ belongs to $ \QUOTf(A)$:
 indeed, these conditions ensure that $F_A^r/V$ is a free $A$-module with Hilbert polynomial $d$.

 Now we prove that the three conditions are also necessary, 
namely that all the modules $M\in \QUOTf(A)$ fullfil them.

By Proposition \ref{star}, it is sufficient to consider 
$ M\in Q_{[i_1,\dots,i_s]}(A)$ for any $[i_1,\dots,i_s]$.

Choose $u_1,\dots, u_d$ as in the proof of Proposition \ref{star}, so that the matrix $\mathbf C(M)$ has an identity submatrix  $\mathrm{Id}_{[i_1,\dots,i_s]} $
 and   take as  $\matrixPM$ the only matrix such that $ \left[\ y u_1  \cdots   yu_d\  \right] =\left[\  x u_1  \cdots x u_d\  \right] \matrixPM$. More explicitely, $\matrixPM$ is constructed  by taking  in the $m$-th block   the columns $2, \dots, i_m+1$, for every $m\leq s$, and the first column of each block from $s+1$ on.   
 In this way, clearly $\matrixPM$ gives the mutiplication by  $y/x$ of each  element of the basis $xu_1, \dots, xu_d$  of  $(F_A^r/V)_{\geq d+1} $.

We conclude the proof observing that the arguments we applied to $(F_A^r/V)_{d+1}$,  also  apply inductively to $(F_A^r/V)_t$ for every $t\geq d+1$.
\end{proof}

We now give few explicit examples for Proposition \ref{basisx}.
\begin{example} \label{ex:casot=5}
We consider $d=5$, $r\geq d$. \\
Assume $M\in Q_{[5]}(A)$: we can take  $x^5e_1,x^{4} ye _1,\dots, xy^{4} e_1$ as a basis of $M$ (as in the proof of Proposition \ref{star}), so that the first block $B_1$ of $\mathbf C(M)$ contains in its first 5 columns an identity matrix. We choose $\ell=x$, $\ell'=y$, and the matrix $\matrixPM$ (multiplication by $y$) is
\[
\matrixPM=\left [
\begin{array}{ccccc}
 0 & 0 & 0 & 0&  b_{1,6} \\
 1 & 0 & 0 & 0 & b_{2,6}  \\ 
 0 & 1 & 0 & 0 & b_{3,6} \\ 
0 & 0 & 1 & 0 &  b_{4,6} \\ 
 0 & 0 & 0 & 1 & b_{5,6} 
\end{array}
\right],
\]
where $[b_{1,6}, b_{2,6}, b_{3,6}, b_{4,6}, b_{5,6}]^T=B_{1,6}^T$ is the last column of the first block of $\mathbf C(M)$.\\
Assume now $M\in Q_{[3,2]}(A)$: we can take  $x^5e_1,x^{4} ye _1,x^3 y^2 e_1, x^5 e_2,x^4 ye_2$ as a basis of $M$, so that the first block $B_1$ of $\mathbf C(M)$ starts with the first  three columns of the identity matrix, and the block $B_2$ starts with the 4-th and 5-th columns of the identity matrix.  We choose $\ell=x$, $\ell'=y$, and the matrix $\matrixPM$ (multiplication by $y$) is in this case
\[
\matrixPM=\left [
\begin{array}{ccccc}
 0 & 0 & b_{1,4} & 0&b_{1,9}\\
 1 & 0 & b_{2,4}  & 0 & b_{2,9}\\ 
 0 & 1 &  b_{3,4}& 0 & b_{3,9}\\ 
0 & 0 &   b_{4,4} & 0& b_{4,9}\\ 
 0 & 0 &  b_{5,4} & 1 &b_{5,9}
\end{array}
\right],
\]
where $[b_{1,4}, b_{2,4}, b_{3,4}, b_{4,4}, b_{5,4}]^T=B_{1,4}^T$ and  $[b_{1,9}, b_{2,9}, b_{3,9}, b_{4,9}, b_{5,9}]^T=B_{2,3}^T$.

Assume now $M\in Q_{[1,1,1,1,1]}(A)$: we can take  $x^5e_1,x^5e_2, x^5e_3, x^5e_4, x^5e_5$ as a basis of $M$, so that the $j$-th column of the $j$-th block of $B_j$ of $\mathbf C(M)$ is the $j$-th column of the identity matrix, $j=1,\dots ,5$.  We choose $\ell=x$, $\ell'=y$, and the matrix $\matrixPM$ (multiplication by $y$) is in this case
\[
\matrixPM=\left [
\begin{array}{ccccc}
b_{1,2} & b_{1,8}&b_{1,14} & b_{1,20}& b_{1,26}\\
b_{2,2} & b_{2,8}&b_{2,14} & b_{2,20}& b_{2,26}\\
b_{3,2} & b_{3,8}&b_{3,14} & b_{3,20}& b_{3,26}\\
b_{4,2} & b_{4,8}&b_{4,14} & b_{4,20}& b_{4,26}\\
b_{5,2} & b_{5,8}&b_{5,14} & b_{5, 20}& b_{5,26}\\
\end{array}
\right],
\]
where $[b_{i,2+(\ell-1)j}]_{i=1,\dots,5}^T=B_{1,(\ell-1)j}^T$ is the 2-nd column of the $\ell$-th block of $\mathbf C(M)$, $\ell=1,\dots,5$. Observe that in this case $\matrixPM$ does not contain any column of the submatrix $\mathrm{Id}_{[1,\dots,1]}$.
\end{example}

\begin{corollary} \label{cor:isom}   For every  $[i_1,\dots,i_s]$ as above, the open subset  $\mathbb Q_{[i_1,\dots,i_s]}$ of $\QUOT$  is isomorphic to the affine space $\mathbb A^{dr} = \Spec (k[w_{h,m}]_{ h=1, \dots, d, \ m=1, \dots, r})$.

An explicit isomorphism $ \phi_{[i_1,\dots,i_s]}$  can  be obtained associating  to  $w_{1,m}, \dots, w_{d,m}$ the entries of the first  column in the $m$-th block   that does not belong to the chosen identity submatrix $\mathrm{Id}_{[i_1,\dots,i_s]}$.
\end{corollary}
\begin{proof}  Let $A$ be any $k$-algebra. We now prove that given description of  $ \phi_{[i_1,\dots,i_s]} $ gives    a bijective  function between $  Q_{[i_1,\dots,i_s]}(A) $  and   $\mathrm{Hom}_A(A[w_{h,m}]_{h=1, \dots, d, \ m=1, \dots, r},A)$.

First we observe that the columns to which we associate the variables $w_{1,m}, \dots, w_{d,m}$    for $m=1, \dots, s$ allow us to construct the matrix $\matrixPM$ of Proposition \ref{basisx}, while for  $m>s$  we  associate to  $w_{1,m}, \dots, w_{d,m}$ the entries of the first column of the $m$-th block.

Therefore, $ \phi_{[i_1,\dots,i_s]}(A)$(M) encodes both $\matrixPM$ and the initial column of each block:   by repeatedly multiplying  the first column of each block by $\matrixPM$  we get the other columns of that block. Therefore, the map $ \phi_{[i_1,\dots,i_s]}(A)$  is  well defined and  injective.

Moreover,  $ \phi_{[i_1,\dots,i_s]}(A)$   is surjective since,   by  Proposition \ref{basisx},  we may freely choose in $A$ the images of the variables $w_{h,m}$ and get a module $M$ in  $ Q_{[i_1,\dots,i_s]}(A)$. 
\end{proof}

\begin{corollary}\label{cor:smooth}
$\QUOT$ is smooth, rational of dimension $dr$.
\end{corollary}

\section{Hilb-support of  points of $\QUOT$}

This short section is dedicated to studying 
  the support of every module $M\in \QUOTf(A)$ in order to associate to it an element of  $\HILBf(A)$.

We have already observed that the support of $M$ is zero-dimensional of lenght  $\leq d$.  The following example shows that the inequality  can be strict.

\begin{example}\label{exAnn} Let $A=k$   and let $M$ be the degree $2$ component of 
$k[x,y]/(y) e_1\oplus k[x,y]/(y)e_2$. 
Then  $M\in {{Quot}^2_{\mathcal  F^2/  \mathbb P^1/  k}}(k)$
  and  its support  is the length 1 subscheme of $\PP^1$ defined by $y=0$.
\end{example}

\begin{notation}
For every   module   $M=(F_A^r/V)_2$ in  $Q_{[i_1,\dots,i_s]}^\ast(A)$, let  
$u_1,\dots, u_d$  the $A$-basis,  $\ell$, $\ell' $ be the basis of $ k[x,y] $    and  $\matrixPM$ be  the matrix  of 
  Proposition \ref{basisx}.    We will denote by  $\chi_{\matrixPM}(T)$ 
 the characteristic polynomial $(-1)^d(T^d- \tr(\matrixPM)T^{d-1}+\dots +\det(\matrixPM))$ of $\matrixPM$ and  by $\mathcal A_M$ the ideal in $S_A$ generated by $ (\ell')^d\chi_{\matrixPM}( \ell'/ \ell) $. 
\end{notation}

\begin{proposition} \label{propannullatore} 
Let  $A$ be a $k$-algebra and $M=(F_A^r/V)_d $ be any module in $Q^\ast(A)$. Then
\begin{enumerate}[\quad \rm 1)]
\item  \label{annullatore1}the ideal $\mathcal A_M$ only  depend on $M$; 
\item   \label{annullatore2}  $\mathcal A_M \subseteq \Ann_{S_A}  (F_A^r/V)  $; 
\item  \label{annullatore3}  $\mathcal A_M =\Ann_{S_A}  (F_A^r/V)  $ if, and only if   $M\in Q^\ast_{[d]}(A)$.
\end{enumerate}
\end{proposition}
\begin{proof}
As the three properties are  local, we may assume that $A$ is local. 

\begin{enumerate}[\quad \rm 1)]
\item  If $M\in Q_{[i_1,\dots,i_s]}(A)\cap Q_{[i'_1,\dots,i'_{s'}]}(A)$, we can construct a generator of $\mathcal A$ 
 considering $M$ as an element of either $ Q_{[i_1,\dots,i_s]}(A)$ or of $ Q_{[i'_1,\dots,i'_{s'}]}(A)$. In both cases $\ell=x$, $\ell'=y$ and the two matrices $\matrixPM$, $P'(M)$ are associated by the invertible matrix of basis  change from  $[u_1, \dots u_d]$ to   $[u'_1, \dots, u'_d]$, so that they have the same characteristic polynomial.

The same happens if we modify the linear basis for $k[x,y]$. It is sufficient to check this fact considering  the following  two special cases. 

If  $\ell_1=x, \ell'_1=y$ and $\ell_2=y, \ell'_2=x$, then the correspondng matrices $P_1(M)$ and $P_2(M)$ are inverse, namely $P_2(M)=(P_1(M))^{-1}$, so that  $  x^d \det(P_1(M)-\frac{y}{x} I)$ and $y^d\det((P_2(M)) ^{-1}-\frac{x}{y} I)$ are equal, up to a sign.

If  $\ell_1=x, \ell'_1=y$ and $\ell_2=x, \ell'_2=y+cx$ with $c\in k$, then the  correspondng matrices $P_1(M)$ and $P_2(M)$ satisfy the equality  $P_2(M)=P_1(M)+cI$, so that  $  x^d \det(P_1(M)-\frac{y}{x} I)=( y+cx)^d(\det(P_2(M) -\frac{y+cx}{x} I)$.

\item     The result is a straightforward consequence of what observed in the proof of Corollary  \ref{cor:isom} and of the well known fact that every matrix is a root of its characterstic polynomial, so that $\chi_{\matrixPM}(\matrixPM)=0$. By  item \ref{annullatore1}) we may born ourselves to the case $\ell'=y$, $\ell=x$.

For every $i=1, \dots, r$  let $B_{i,j}$ the  $j$-th column of the $i$-th block of $\mathbf C(M)$. Then  $B_{i,j}=	\matrixPM^{j-1}B_{i,1}$. 
 If $\chi_{\matrixPM}(T)=a_1  +a_2T +\dots + a_dT^{d-1}+T^d$, then    
\begin{multline*}   a_1 B_{i,1}  +a_2 B_{i,2}+\dots + a_d B_{i,d}+B_{i,d+1}=\\
= a_1 B_{i,1}  +a_2P B_{i,1}+\dots + a_d \matrixPM^{d-1}B_{i,1}+\matrixPM^dB_{i,1}=  \chi_{\matrixPM}(\matrixPM) B_i=0.
\end{multline*}

Therefore,  in $M$ we have  $x^d\chi_{\matrixPM}(y/x)e_i=(a_1x^d  
+a_2x^{d-1}y  +\dots + a_dxy^{d-1} +y^d)e_i=0$, for every $i=1,\dots, r$, hence    $\mathcal A_M\subseteq \Ann_{S_A}(F^r_A/V)$. 1

\item
 Both ideals  $\mathcal A_M$  and $\Ann_{S_A}(F^r_A/V)$   are homogeneous  and  every   polynomial $f\in S_{A,m} $ , $m\geq d$, can be written as  $(a_1x^d   + +a_2x^{d-1}y  +\dots + a_dxy^{d-1} +y^d)g+    x^{m-d}(c_1x^d+c_2x^{d-1}y+ \dots+c_{d} xy^{d-1})  $ for some $c_1, \dots, c_{d}\in A$. Then,  if not empty, the difference  $ \Ann_{S_A}  (F^r_A/V)\setminus  \mathcal A_M$ should contain some element of the tipe $c_1x^d+c_2x^{d-1}y+ \dots+ c_{d}x y^{d-1}) $: recall  that in  Proposition \ref{basisx}  we proved that the multiplication by $x$ in $(F_A^r/V)_{\geq d}$  is injective.

If $F_A^r/V\in  Q_{[d]}(A)$, then $u_1=x^de_1, \dots, u_d=xy^{d-1}e_1$  is a basis of $M$, so that  $(c_1x^d+c_2x^{d-1}y+ \dots+ c_{d}x y^{d-1}) e_1=c_1u_1+c_2u_2+ \dots +c_du_d$ vanish   if, and only if,  $c_1=c_2=\dots =c_d=0$. Therefore $ \Ann_{S_A}  (F^r_A/V)= \mathcal A_M$.
\end{enumerate}
\end{proof}

\begin{definition}\label{def:hilbsupp}
Let $A$ be any $k$-algebra and $M$ be any element of $Q^\ast(A)$. 
We call  {\it Hilb-support} of $M$ the scheme  $\Proj(A[x,y]/ \mathcal A_M)$. 
\end{definition}

We now motivate the name we choose for this new kind of support of a module.

\begin{lemma}
For any $k$-algebra $A$, if  $M$ is an element of $Q^\ast(A)$,  then 
the schematic support $Z$ of $M$ coincides with the  Hilb-support if,  and only if $M\in Q^\ast_{[d]}(A)$.

In particular, for $A=k$,  if $M$ is any module in $\QUOTf(k)$ whose schematic support  is a zero-dimensional scheme $Z$ formed by $d$ distinct $k$-point of $\PP_1$, then $Z$ is also the Hilb-support.
\end{lemma}

\begin{proof}
In order to prove this statement it is sufficient to improve the proof of  Proposition \ref{star}. Let $M$ be  a module $Q^\ast(A)$:  as $k$ is an infinite field,  up to a general change of the basis  $e_1, \dots, e_r$  with coefficients in $k$, we can assume that  the support of $M$ coincides with the support of the submodule of $M$ generated by $e_1$.

Now, if the support of $M$ does not have degree $d$, then it is not the Hilb-support of $M$. Otherwise $M$ belongs to  $Q^\ast_{[d]}(A)$  and we apply Proposition  \ref{propannullatore} .
\end{proof}

We now give a concrete insight to the notion of Hilb-support, when $A=k$.

Observe that if we take $M_1 \in {Quot}^{d_1}_{\mathcal  F^{r_1}/  \mathbb P^1/k}(k)$, $M_2 \in{Quot}^{d_2}_{\mathcal  F^{r_2}/  \mathbb P^1/k}(k)$, then the Hilb-support of the module $M_1\oplus M_2 \in  {Quot}^{d_1+d_2}_{\mathcal  F^{r_1+r_2}/  \mathbb P^1/k}(k) $ is the sum of the Hilb-support of the two modules.

Let $A=k$, $M\in \QUOTf (k)$ and let $Z=\{R_1,\dots, R_n\}$ be the set-theoretical support of $M$.  We denote by $I_{R_j}$ the maximal saturated ideal in $S$ defining the point $R_j\in \mathbb P^1$.
We consider $\mu_j:= \dim_k (M\otimes_S S/I_{R_j})$. In this case, the {\it Hilb-support} of $M$ is $\mu_1 R_1+\cdots+\mu_n R_n\in \HILBf(k)$.

\begin{definition}\label{def:xi}

We define  a morphism of schemes $\xi  \colon \QUOT \to \HILB $  using open subfunctors. We first define $\xi$  on 
$Q_{[i_1,\dots,i_s]}$, for any  $[i_1,\dots,i_s]$:\\
  for every $k$-algebra $A$,  
\[
\xi(A) \colon Q_{[i_1,\dots,i_s]}(A) \to \HILBf(A), M\mapsto \xi(A)(M)=\mathcal A_M,
\]
The map  $\xi$ is well defined thanks to Proposition \ref{propannullatore}.
\end{definition}

\begin{remark}
By definition, the map $\xi$ associate to every $k$-point $M$ of $\QUOT$  its Hilb-support $Y$, that does not always coincide with  the schematic support $Z$ of $M$.   Therefore the fibre of $\xi$ over a $k$-point $Y$ of $\HILB$ is in general a proper subscheme of $\QUOTZ$, as shown by the following example.  However, we can describe the fibre of $\xi$ over any point $Y$ of $\HILB$ using Quot schemes over $0$-dimensional schemes of $\PP^1$. 
\end{remark}

\begin{example}
For $r=2$, $d=2$, consider $Y=R+R'$, with $R\neq R'$. Every module whose Hilb-support is $Y$ is completely determined by its restrictions to $R$ and to $R'$;  the condition on the Hilb support requires that each restriction is locally free with rank 1. Therefore,   the fiber of $\xi$ over $Y$ is irreducible and isomorphic to $\mathbb P^1\times \mathbb P^1$. However $\mathrm{Quot}^2_{\mathcal  F_Y/  Y/   k}$ has also two other isolated points, corresponding to the modules $M_1=\mathcal O_R\oplus \mathcal O_R$ and $M_2=\mathcal O_{R'}\oplus \mathcal O_{R'}$, whose schematic support is $R\subsetneq Y$ (resp. $R'\subsetneq Y$). Observe that $\xi(M_1)=2R$ and $\xi(M_2)=2R'$.
\end{example}

\begin{proposition}\label{prop:fibers} Let $Y$ be a degree $d$ 0-dimensional subscheme of $\PP^1$

\begin{enumerate}[\quad \rm1)]
\item If $Y=dR$, then $\xi^{-1}(Y)=\mathrm{Quot}^d_{\mathcal  F_Y/ Y /   k}$. 
\item If $Y=Y_1+ \dots Y_n$ with $Y_1=t_1 R_1, \dots,  Y_n=t_n R_n $,  where $R_1, \dots R_n$ are distinct points, then 
$\xi^{-1}(Y)=\mathrm{Quot}^{t_1}_{\mathcal  F_{Y_1}/ Y_1/   k}\times \dots \times \mathrm{Quot}^{t_n}_{\mathcal  F_{Y_n}/ Y_n/   k}$. 
\end{enumerate}
\end{proposition}
\begin{proof}
In the first case, a $k$-point of $\xi^{-1}(Y)$ must have set  theoretical support  $R$   and have Hilbert polynomial $d$. As already observed, this is equivalent to require that the Hilb-support  is $dR$.

In the second case every module $M$  in the fibre of $Y$ is determined by its restriction $M_i$ to tech  points  $R_i$; by definition of Hilb-support the Hilbert polynomial of  the restriction to $R_i$ must be exactly $t_i$;  hence,  by the previous item,  $M_i$  is a point of $\mathrm{Quot}^{t_1}_{\mathcal  F_{Y_1}/ Y_1/   k}$ .
\end{proof}

Due to the previous result, in next section we study the Quot schemes over  a $0$-dimensional scheme of  $\PP^1$  whose set theoretical support is a single point.

\section{Quot over a 0-dimensional scheme supported on a single point}


In the present section we study the modules belonging to $\QUOTZ$, Quot scheme defined on $Y$, 0-dimensional scheme of $\mathbb P^1$ supported at a single point.

We denote by  $I_Y\subset S$ the saturated ideal such that $\Proj(S/I_Y)=Y$,  and by $\mathcal O_Y$ the structural sheaf of  $\Proj(S/I_Y)$. We set  $\mathcal F^r_Y :=\oplus^r \mathcal O_Y= \mathcal O_{Y}e_1 \oplus \cdots \oplus \mathcal O_{Y}e_r$.

\begin{definition}  We will denote by $Q^\ast_{[i_1,\dots,i_s],Y}$ (resp. $ \mathbb Q^\ast_{[i_1,\dots,i_s],Y}$) the open subfunctors (resp. subschemes) of  $\QUOTZf$ (of $\QUOTZ$) defined similarly to $Q^\ast_{[i_1,\dots,i_s]}$ (resp. $ \mathbb  Q^\ast_{[i_1,\dots,i_s]}$) in ${Quot}^t_{\mathcal  F^r/  \mathbb P^1/  k}$ (resp. $\mathrm{Quot}^t_{\mathcal  F^r/  \mathbb P^1/  k}$).  
\end{definition}

Every statement concerning an open subset $\mathbb Q^\ast_{[i_1,\dots,i_s]}$ of $\mathrm{Quot}^t_{\mathcal  F^r/  \mathbb P^1/  k}$ also apply to the open subset  $\mathbb Q^\ast_{[i_1,\dots,i_s],Y}$ of  $\QUOTZ$, provided that the matrix $\matrixPM$ of $M \in Q^\ast_{[i_1,\dots,i_s]}(A)$  is a root of the equation defining $Y$.  
More explicitely,  consider $Y=tR$ with $R$ a single point.  Assuming by simplicity that   $R$ is defined by $y=0$, the condition we have to impose  is the vanishing of  $\matrixPM^t$. 
Hence, we are interested in studying the following functor.

\begin{lemma}  \label{lem:W}  Let $\matrixPP=\left(w_{ij}\right)$  be  the $t\times t$ matrix with indeterminates entries  and let $\mathcal P$ be the functor from Ring to Sets  $$ A\mapsto \left\{t\times t \hbox{ matrices $P$ with entries in  $A$ such that } P^t=0 \right\}$$
and, for every ring homomorphism $ f\colon A\rightarrow B $

$$  P= \left(a_{ij}\right) \in \mathcal P(A) \mapsto  f(P)=\left(f(a_{ij})\right) \in \mathcal P(B).$$
Then $\mathcal P$ is the functor of points of the subscheme of $\mathbb A^{t^2}=\Spec(k[w_{ij}]) $ defined by the ideal 
$I_{\matrixPP^t}$ generated by the entries of 
$\matrixPP^t$.
\end{lemma}
\begin{proof}
Obvious.
\end{proof}

For every $k$-algebra $A$, similarly to what we did in Proposition \ref{star}  for  $\mathbb Q_{[i_1,\dots,i_s]}$,  we can construct    $\mathbb Q_{[i_1,\dots,i_s],Y}$,   as  a subscheme of the  affine scheme $\mathbb A^{tr}$. Consider now the matrix $\matrixPP_{[i_1,\dots,i_s]}$ having the same shape as the matrix $\matrixPM$ for $M\in Q_{[i_1,\dots,i_s],Y}(A)$ (see Proposition \ref{basisx}),  and whose entries not in the columns of $\mathrm{Id}_{[i_1,\dots,i_s]}$ are parameters $w_{h,m}$, $h=1,\dots,t $, $m=1,\dots, r$. Then $\mathbb Q_{[i_1,\dots,i_s],Y}$ is defined  by
\begin{equation}\label{varw}
\Spec\left(k[w_{h,m}]_{ h=1, \dots,   t, \ m=1, \dots,   r}/I_{\matrixPP_{[i_1,\dots,i_s]}^t}\right),
\end{equation}
where, depending on the chosen $[i_1,\dots,i_s]$, the indeterminates  $w_{h,m}, \ h=1, \dots,   t, \ m=1, \dots,  r$  correspond to the suitable entries of $\matrixPP_{[i_1,\dots,i_s]}$ (as described in in Corollary \ref{cor:isom}),  and $I_{\matrixPP_{[i_1,\dots,i_s]}^t}$   the ideal generated by the entries of the matrix $\matrixPP_{[i_1,\dots,i_s]}^t$.

Observe that, if $t\leq r$, then $\matrixPP_{[1,\dots,1]}=\matrixPP$ (there are no identity columns in $\matrixPP_{[1,\dots,1]}$). 
The scheme representing the functor $\mathcal P$ associated to $\matrixPP_{[i_1,\dots,i_s]}$ is geometrically non-trivial. Indeed, even for $t=2$ and $\matrixPP_{[1,1]}$, the scheme representing $\matrixPP_{[1,1]}$ has an isolated singular component and an  embedded component.

\begin{example} \label{ex:d2}Consider $t=2$ and $\matrixPP_{[1,1]}=\matrixPP$. We can explicitly write $I_\matrixPP\subset k[w_{1,1},w_{1,2},w_{2,1},w_{2,2}]$:
$$I_\matrixPP=(w_{1,1}^2+w_{1,2}w_{2,1}, (w_{1,1}+w_{2,2})w_{1,2}, (w_{1,1}+w_{2,2})w_{2,1}, w_{2,1}w_{1,2}+w_{2,2}^2).$$
Computing a primary decomposition of it, we see that  the scheme $W:=\Spec(k[w_{h,m},h=1,2,m=1,2]/I_\matrixPP)$ is   formed by an irreducible isolated $2$-dimensional  component $C$  given by $$I_C=(w_{1,1}+w_{2,2},  w_{1,1}w_{2,2}-w_{1,2}w_{2,1})$$ where $w_{1,1}+w_{2,2}=\tr(\matrixPP)$  and   
$ w_{1,1}w_{2,2}-w_{1,2}w_{2,1}=\det(\matrixPP)$ and an embedded component at the point of $C$ corresponding to the null matrix, given for instance by the ideal 
$$( w_{1,2},w_{2,1},w_{1,1}^{2},w_{2,2}^2)
$$
(computations of a  primary decomposition have been  made by the Maple software).

Note that  $I_C$ is also generated by $ w_{1,1}+w_{2,2} $ and  $w_{1,1}^2+w_{1,2}w_{2,1}$, hence $C$ is a quadric cone  in a hyperplane of  $\mathbb A^4$ and the vertex corresponds to the null matrix.

\end{example}

We will now describe $\QUOTZ$with $Y=t R$. 
 For every $M$ free module belonging to $\QUOTZf(A)$, if $M$ belongs to several open subsets $Q^\ast_{[i_1,\dots,i_s],Y}(A)$, we consider the one given by the largest list of indexes according to lex, because this open subsets allows us to characterize more precisely the geometric features of $\QUOTZ$ at $M$.
   In particular,  we  will  be interested in  $Q^\ast_{[t],Y}$, the one having the lex-largest list of indices.

\begin{theorem}\label{thm:sing}
Let  $r, t\geq2$,   and let $Y$ be the $0$-dimensional degree t  scheme $tR$ in $\mathbb P^1$.  \\
Then  $\QUOTZ$ is an irreducible non-reduced  scheme  of dimension $t(r-1)$. Its singular locus    is contained  in the complementary of $\mathbb Q^\ast_{[t],Y}$, 
while  $\mathbb Q^\ast_{[t],Y}$ is 
 locally isomorphic to   $\AA^{t(r-1)}$. 
\end{theorem}
 
\begin{proof} 
As always, we assume that $R$ is the point of $\PP^1$ given by $y=0$, so that $tR$ is given by $y^t=0$.

We first prove the statement concerning  $\mathbb Q^\ast_{[t],Y}$  showing that  $\mathbb Q_{[t],Y}$ is isomorphic to $\mathbb A^{r(t-1)}$.
Arguing  as in the proof of Corollary \ref{cor:isom}, we see that the modules in $\mathbb Q_{[t]}(A)$ are parameterized by  $t\cdot r$ parameters (the entries of the column $t+1$ in the first block and the entries of  first column in each   following block of $\mathbf C(M)$). Moreover, the matrix $\matrixPP_{[t]}$ is formed by the columns $2, \dots, t$ of the $t\times t$ identity matrix and by the last  column of the first block of $\mathbf C(M)$, as follows
$$\matrixPP_{[t]}= \left[ \begin {array}{cccccc} 
0&0&\dots  &0&0&{w_{1,1}}
\\1&0&\dots&0  &0&{w_{2,1}}
\\ 0&1&\dots &0&0&{ w_{3,1}}
\\ \vdots &\vdots &\ddots &\vdots &\vdots&\vdots 
\\0&0&\dots &1&0&{w_{t-1,1}}
\\ 0 &0&\dots &0&1&{w_{t,1}}
\end {array} \right] 
$$
where the symbols for the   entries of the last column are choosen coherently with  \eqref{varw}.

As in \eqref{varw},  the  ideal of $k[w_{h,m}]_{h=1, \dots, t, \ m=1, \dots, r}$ defining   $\mathbb Q_{[t],Y}$ in  $\mathbb Q_{[t]}\simeq \mathbb A^{rt}$ is generated by the entries of   $\matrixPP_{[t]}^t$, that only involve the last  variables $w_{1,1}, \dots, w_{t,1}$. The multiplication of a matrix $D$  by $\matrixPP_{[t]}$ gives a matrix whose column $i$ is the column $i+1$ of $D$, for every $i=1, \dots , d-1$;  hence    the first column of $\matrixPP_{[t]}^t$ coincides with the last one in $\matrixPP_{[t]}$.   Therefore the ideal $I_{\matrixPP^t_{[t]}}$ contains ( hence    coincides with)  the    ideal generated by  $w_{1,1}, \dots, w_{t,1}$, so that  $\mathbb Q_{[t],Y}$ is isomorphic to $\AA^{t(r-1)}$.

In order to prove that $\mathbb Q_{[t],Y}^\ast$ is irreducible, it is sufficient to observe that, for every pair of $k$-points of $\mathbb Q_{[t],Y}^\ast$  we may performe a sufficiently  general change of the basis $e_1, \dots, e_r$ such  that both  of them  belong to $\mathbb Q_{[t],Y}$, which is  is irreducible, as we have just proved.

Now we prove that   $\QUOTZ$ has embedded components; by what we have just proved they must be supported on $\QUOTZ\setminus\mathbb  Q^\ast_{[t]} $.  To this aim we consider 
 the open subscheme 
$\mathbb  Q_{[t-1,1],Y}$ of $\QUOTZ$ and consider it as a closed subscheme of $\mathbb  Q_{[t-1,1]}$

Following the description   given in Proposition \ref{star},   $\mathbb  Q_{[t-1,1]}$ is isomorphic to $\mathbb A^{rt} $,  the  matrix $\matrixPP_{[t-1,1]}$ is as follows
\begin{equation} \label{[t-1,1]}  \matrixPP_{[t-1,1]}= \left[ \begin {array}{cccccc} 
0&0&\dots  &0&{w_{1,1}}&{w_{1,2}}
\\1&0&\dots&0  &{w_{2,1}}&{w_{2,2}}
\\ 0&1&\dots &0&{ w_{3,1}} &{w_{3,2}}
\\ \vdots &\vdots &\ddots &\vdots &\vdots&\vdots 
\\ 0&0&\dots &1&{w_{t-1,1}}&{w_{t-1,2}}
\\ 0 &0&\dots &0&{w_{t,1}}&{w_{t,2}}
\end {array} \right] 
\end{equation} 
and the  ideal $J$ defining  $\mathbb Q_{[t-1,1],Y}$ as a subscheme of $\mathbb Q_{[t-1,1]}
\simeq \mathbb A^{tr} $  is generated  by  the entries   of  $\matrixPP_{[t-1,t]}^t$. Note that   $\matrixPP_{[t-1,t]}^t=0$ gives  constraints   only involvig  the variables $w_{i,j}$ with $j=1,2$.  

We observe that the matrix $\matrixPM$ associated to any  $k$-point of $\mathbb Q_{[t-1,1],Y}$ is  given by setting   $w_{i,j}=c_{i,j}\in k$ for every  $i=1, \dots, t$, $j=1,2$  in $\matrixPP_{[t-1,t]}$. $\matrixPM$ has characteristic polynomial $T^t$, so that in particular it must have  a null trace, namely  $c_{1,t-1}+c_{2,t}=0$, so that the polynomial $w_{t-1,1}+w_{t,2}$ vanishes on the support of $\mathbb Q_{[t-1,1],Y}$, namely  $w_{t-1,1}+w_{t,2} \in\sqrt J$.  

An easy check  shows that  in $\matrixPP_{[t-1,1]}^2$ (and then in any following power of $\matrixPP_{[t-1,1]}$)   the variable  $w_{t,2}$  appears only in monomials of degree  larger than $1$, so that  $w_{t-1,1}+w_{t,2}\notin J$.  
 Therefore, $\mathbb Q_{[t-1,1],Y}$   is not reduced.  
 
 Observe that $\mathbb Q_{[t-1,1],Y}$  intersects properly $\mathbb Q_{[t],Y}$: consider for instance the quotient module of $F_Y$ given by the relations $x^t e_2=(x^t+y^t)e_1$ and $e_j=0$, for $j\geq 3$. This means that the non-reduced structure is only given by embedded components.
 
The  above   arguments 
  also prove that  $\sqrt J$ is generated by the non-leading coefficients of the characteristic polynomial of $\matrixPP_{[t-1,1]}$ and  it is a prime ideal. 

We can then conclude that $\mathbb Q_{[t-1,1],Y}$, being irreducible, non reduced,  and  generically smooth,  has  (at least) one embedded component.

Finally, we prove that the support of $\mathbb Q_{[t-1,1],Y}$ is singular.   
Let $M$ be any $k$ point of  the support of $\mathbb Q_{[t-1,1],Y}$  corresponding to setting $ w_{i,j}=0$ for every $j=1,2$ (and freely choosing values in $k$ for the remaining  $(r-1)t$  variables).
 Now we compute the Zariski tangent space at  $M$  to the support of $\mathbb Q_{[t-1,1],Y}$  in $\mathbb A^{tr}$  taking the linear component of the  non-leading  coefficients of the characteristic polynomial of $\matrixPP_{[t-1,1]}$. The linear part of the coefficient of $T^{i} $ is $w_{i,1}$ for $i=1, \dots, t-1$, while for $i=0$ the coefficient is the  determinant  $w_{t-1,1} w_{t,2}-w_{t,1}w_{t-1,2}$, which is  homogeneous of degree $2$. Therefore, the Zariski tangent space at $M$ has codimension $t-1$, while the codimension of  $\mathbb Q_{[t-1,1],Y}$ in $\mathbb A^{tr}$ is $t$.  
\end{proof}

We can now prove the main result of the paper.

\medskip

\noindent {\it Proof of Theorem \ref{thm:main}.}
It is sufficient to sum up the results of Proposition \ref{prop:fibers} and Theorem \ref{thm:sing}.
\qed

\medskip

 We conclude the paper with two examples.

The following example concerns the case  $r\geq t\geq2 $  and describes   singular points for $\QUOTZ$ which is not included in those considered in the proof of the previous result.

\begin{example}\label{ex:[1,..,1]} Let $r\geq t\geq2$ and let $Y=tR$; assume that $R$ is the point of $\PP^1$ given by $y=0$.

 If (and only if) this condition on  $t,r$ is fullfilled, we can find    in $\QUOTZ$ points corresponding to modules $M$ whose schematic   support is $R$. They are the points of $\mathbb Q_{[1,\dots,1],Y}^\ast$  such that $\matrixPM =0$ (while all the points in  $Q_{[i_1,\dots,i_n],Y}^\ast (A)$   with $i_1>1$ it holds    $\matrixPM^{i_1-1}\neq 0$, hence $\matrixPM \neq 0$).  In this example we describe these points and prove that they are all singular  in $\QUOTZ$.

Making reference to Corollary \ref{cor:isom} (and also to Example \ref{ex:casot=5}),  a  module $M$  in $Q_{[1,\dots,1]}(A)$ is determined by  the second columns of the blocks $B_1,\dots, B_t$ and the first columns of the following blocks of $\mathbf C(M)$. The matrix $\matrixPP_{[1,\dots,1]}=\matrixPP$ 
 has the following form 
\begin{equation} \label{[1,..,1]}  \matrixPP=\left[ \begin {array}{cccc} w_{{1,1}}&w_{{1,2}}&\dots &w_{{1,t}}
\\ w_{{2,1}}&w_{{2,2}}&\dots &w_{{2,t}}
\\ w_{{3,1}}&w_{{3,2}}&\dots &w_{{3,t}}
\\ \vdots &\vdots &\ddots &\vdots 
\\ w_{{t,1}}&w_{{t,2}}&\dots &w_{{t,t}}
\end {array} \right] \end{equation}

Furthermore, the affine scheme $\mathbb Q_{[1,\dots,1],Y}$ is defined by the ideal generated by the entries of the matrix $\matrixPP^t$ in the polynomial ring $k[w_{i,j}]$.   In $\mathbb A^{t^2}$, the entries of $\matrixPP^t$ define the so-called {\it null cone}. The arguments  used in the proof of Theorem \ref{thm:sing}  
 prove that it is irreducible and not reduced of dimension $t(t-1)$.
Hence, $\mathbb Q_{[1,\dots,1],Y}$ is the product of this null cone with an affine space of dimension $t(r-t)$.

Observe that  $Q_{[1,\dots,1],Y}(A)$ properly intersects $Q_{[t],Y}(A)$; consider for instance the module $M$ given by the quotient of $F_Y$ with the relations $x^t e_i-(x^t +x^{t-i+1}y^{i-1})e_1$, for $i=2,\dots,t$, and $e_i=0$ for $i\geq t+1$. This means that the non-reduced structure is only given by embedded components.

The locus in $\mathbb Q_{[1,\dots,1],Y}$    we are now  focusing on  is defined  by $\matrixPP=0$, hence it is isomorphic to $\mathbb A^{t(r-t)}$. All the points in this locus are singular points of $\QUOTZ$. Indeed, considering  the embedding  $\mathbb Q_{[1,\dots,1],Y}\hookrightarrow \mathbb Q_{[1,\dots,1]}\simeq\mathbb A^{tr}$,  the Zariski tangent space at each such point has codimension $t^2$, while the codimension of $ \mathbb Q_{[1,\dots,1],Y}$ is $t$. 

\end{example}

\begin{example}  For $r=t=2$ and $Y=2R$,  $\mathrm{Quot}^2_{\mathcal  F_Y/ Y /   k}$  has a unique singular point, the vertex of the cone described in Example \ref{ex:d2}.  Note that for $t=2$,  $\mathbb  Q_{[1,1],Y}=\mathbb  Q_{[t-1,1],Y}$, namely  Example \ref{ex:d2} is both  a special case of   Example \ref{ex:[1,..,1]} and of the case $\mathbb  Q_{[t-1,1],Y}$ in the proof of Thereom \ref{thm:sing}.  Indeed,   the  vertex of the cone  is actually the unique point in $\mathrm{Quot}^2_{\mathcal  F_Y/ Y /   k}  \setminus\mathbb  Q^\ast_{[2],Y}$, and it is not only singular for the support of  $\mathrm{Quot}^2_{\mathcal  F_Y/ Y /   k}$, but  at this point is supported   the only  embedded component  $\mathrm{Quot}^2_{\mathcal  F_Y/ Y /   k}$.  
\end{example}

\section*{Acknowledgements}
 The authors are grateful to Roy Skjelnes for several discussions and his insight on the topic of the present paper. The authors also
 thank Federica Galluzzi for discussions on the Hilbert-Chow morphism, that inspired the definition of the map $\xi$ of the present paper, and Qiacohu Yuan for useful suggestions on the null-cone of nilpotent matrices.

\end{document}